\documentclass{amsart}

\usepackage{amsthm,amssymb,amsmath,hyperref}
\usepackage{graphicx}
\usepackage{color}
\usepackage{tikz}
\usepackage[utf8]{inputenc}
\usepackage{mathtools}

\theoremstyle{plain}
\newtheorem{thm}{Theorem}[section]

\newtheorem{cor}[thm]{Corollary}

\newtheorem{prop}[thm]{Proposition}

\theoremstyle{definition}
\newtheorem{defn}[thm]{Definition}

\theoremstyle{remark}
\newtheorem{rem}[thm]{Remark}
\newtheorem{exa}[thm]{Example}
\theoremstyle{plain}

\numberwithin{equation}{section}

\newcommand{\del}{\delta}

\newcommand{\R}{{\mathbb R}}

\newcommand{\N}{{\mathbb N}}

\newcommand{\Z}{{\mathbb Z}}

\newcommand{\calC}{{\mathcal C}}
\newcommand{\calD}{{\mathcal D}}

\newcommand{\calG}{{\mathcal G}}

\newcommand{\calL}{{\mathcal L}}

\newcommand{\calN}{{\mathcal N}}

\newcommand{\calV}{{\mathcal V}}

\newcommand{\bfVa}{{\Vec{\textbf{a}}}}

\def\udot#1{\ifmmode\oalign{$#1$\crcr\hidewidth.\hidewidth
    }\else\oalign{#1\crcr\hidewidth.\hidewidth}\fi}

\def\R{\mathbb{R}}
\def\Z{\mathbb{Z}}
\def\T{\mathbb{T}}

\begin{document}
	
\title[]{Sharp Mei's lemma with different bases}
\author{Theresa C. Anderson}
\author{Bingyang Hu}

\address{Theresa C. Anderson: Department of Mathematics, Purdue University, 150 N. University St., W. Lafayette, IN 47907, U.S.A.}%
\email{tcanderson@purdue.edu}

\address{Bingyang Hu: Department of Mathematics, Purdue University, 150 N. University St., W. Lafayette, IN 47907, U.S.A.}%
\email{hu776@purdue.edu}

\begin{abstract}
In this paper, we prove a sharp Mei's Lemma with assuming the bases of the underlying general dyadic grids are different. As a byproduct, we specify all the possible cases of adjacent general dyadic systems with different bases.  The proofs have connections with certain number-theoretic properties. 

\end{abstract}
\date{\today}

\thanks{The first author is funded by NSF DMS 1954407 in Analysis and Number Theory. }

\maketitle

\section{Introduction}
The purpose of this paper is to give an optimal description of the adjacency of general dyadic systems in $\R^d$ with different bases. The study of describing dyadic systems in a refined manner dates back to the work \cite{Conde} of Conde Alonso, in which, he proved $d+1$ is the optimal number of dyadic systems in $\R^d$ to guarantee adjacency. 

However, Conde Alonso's result only implies the existence of such a collection of $d+1$ dyadic systems. Our goal is to understand for a given collection of $d+1$ dyadic grids (or more general, $n$-dyadic grids), what the necessary and sufficient conditions are so that such a collection is adjacent. In our recent paper \cite{AHJOW} joint with Jiang, Olson and Wei, we answered this question when $d=1$, and later in \cite{AH} we extended this result to higher dimensions by studying the \emph{fundamental structures of $d+1$ $n$-adic systems in $\R^d$}.  Note that in both \cite{AHJOW} and \cite{AH}, the bases of the given $d+1$ grids are the same. 

In this paper, we further generalize these results to the case when the bases of these $d+1$ grids are different. Moreover, we are also able to specify all the possible cases for adjacent systems with different bases. Let us begin with the definition of $n$-adic systems in $\R^d$, which is our main object of interest in this paper. 

\begin{defn}
Given $n \in \N, n \ge 2$, a collection $\calG$ of left-closed and right-open cubes on $\R^d$ (that is, a collection of cubes in $\R^d$ of the form
$$
[a_1, a_1+\ell) \times \dots \times [a_d, a_d+\ell), \quad a_i \in \R, i=1, \dots, d, 
$$
where $\ell>0$ is the \emph{sidelength} of such a cube) is called \emph{a general dyadic grid with base $n$ (or $n$-adic grid)} if the following conditions are satisfied:
\begin{enumerate}
\item [(i).] For any $Q \in \calG$, its sidelength $\ell(Q)$ is of the form $n^k, k \in \Z$;
\item [(ii).] $Q \cap R \in \{Q, R, \emptyset\}$ for any $Q, R \in \calG$;
\item [(iii).] For each fixed $k\in \Z$, the cubes of a fixed sidelength $n^k$ form a partition of $\R^d$.
\end{enumerate}
Note that when $n=2$, the above definition refers to the classical \emph{dyadic system} in $\R^d$, which we denote by $\calD$. 
\end{defn}

An important property for such a structure is the following optimal dyadic covering theorem due to Conde \cite{Conde}, which is also known as the \emph{optimal Mei's lemma}.  

\begin{thm} {\cite[Theorem 1.1]{Conde}} \label{CondeThm}
There exists $d+1$ dyadic grids $\calD_1, \dots, \calD_{d+1}$ (with base $2$) of $\R^d$ such that every Euclidean ball $B$ (or every cube) is contained in some cube $Q \in \bigcup\limits_{i=1}^{d+1} \calD_i$ satisfying that $\textrm{diam}(Q) \le C_d \textrm{diam}(B)$. The number of dyadic systems is optimal. 
\end{thm}

The origin of the adjacency of dyadic systems is obscure but we believe that credit should be given to Okikiolu \cite{Okikiolu} and, for a somewhat weaker version, to Chang, Wilson and Wolff \cite{CWW}. Later in 2013, Hyt\"onen and Per\'ez \cite{HP} proved that Mei's lemma holds $2^d$ dyadic grids with the constant $C_d=6$, and in the same year, Conde Alonso \cite{Conde} showed the optimal number of the dyadic systems needed is $d+1$ but with a larger constant $C_d \simeq d$. Moreover in 2014, Cruz-Uribe \cite{DCU} gave a short proof of Mei's lemma for $3^d$ dyadic grids with a better constant $C_d=3$. We would also refer the reader for \cite{LN, LPW, PW, HK} and the references there in for more detailed information about the development of this property. The adjacency of dyadic systems are crucially used in harmonic analysis (for instance, by Lerner to prove the $A_2$ theorem in \cite{AL}, among many other recent papers on \emph{sparse domination}), functional analysis \cite{C2}, \cite{GJ}, \cite{LPW}, \cite{PW} and measure theory \cite{CP}.

Theorem \ref{CondeThm} motivates the following definition of the adjacency of a collection of $d+1$ general dyadic grids with different bases. Note that the adjacency we are considering in this note is more general than the one in \cite{AH} and \cite{AHJOW}. 

\begin{defn} \label{repdyadic}
Given $d+1$ general dyadic grids $\calG_1, \dots, \calG_{d+1}$, where the base of $\calG_i$ is $n_i$, $i=1, \dots, d+1$, we say they are \emph{adjacent} if for any open cube $Q \subseteq \R^d$ (or any ball), there exists $i \in \{1, \dots, d+1\}$, and $D \in \calG_i$, such that $Q$ is \emph{comparable} to $D$, in the sense that
\begin{enumerate}
    \item [(1).] $Q \subseteq D$;
    \item [(2).] $\ell(D) \le C\ell(Q)$, where the constant $C$ only allows to depend on $n_1, \dots, n_{d+1}$ and $d$, in particular, it is independent of the sidelength of the cubes $Q$ and $D$. 
\end{enumerate}
\end{defn}

The \emph{new feature} of the adjacency of general dyadic systems with different bases comes from the fact that the cubes from different grids living in different generations start interacting with each other, in both \emph{small scale case} (where the cube has sidelength less than or equal to $1$) and \emph{large scale case} (where the cube has sidelength great than $1$). This leads to the fact that some generations of the general dyadic grids make a significant contribution to the adjacency, while some make no contributions. This is quite different from the case considered in \cite{AHJOW} and \cite{AH} where there is only one base; the adjacency there is decided by cubes from all generations in different grids. 

\begin{exa} \label{20200922exa01}
Here is an easy way to produce adjacent general dyadic systems with different bases. 

To start with, we can take a known example of adjacent grids with the same bases (see, e.g., \cite[Page 786--787]{Conde}), and then change bases by deleting specific generations. For example, let $\calD_1$ and $\calD_2$ be two dyadic grids which are adjacent on $\R$, then we define $\calG_1:=\calD_1$ and
     \begin{equation} \label{20200910eq04}
     \calG_2:=\bigcup_{i \in \Z} \left\{ I \in \left(\calD_2 \right)_{4i} \right\}.
     \end{equation} 
Here are some remarks for the above example. 
\begin{enumerate}
    \item [(1).] It is easy to check that $\calG_1$ and $\calG_2$ are also adjacent on $\R$, while $\calG_2$ is of base $16$;
    
    \medskip
    
    \item [(2).] This construction easily suggests to the following fact: the optimal number that is needed to guarantee the adjacency for grids with different bases is also $d+1$;
    
    \item [(3).] It turns out that this ``changing bases trick" is the \emph{only} possible case for adjacent systems with different bases (see, Theorem \ref{mainresult02}). 
\end{enumerate}

Finally, note that not all generations in $\calG_1$ make a contribution to the adjacency. Indeed, let us define
$$
\calG_1':=\bigcup_{i \in \Z} \left\{ I \in \left(\calD_1 \right)_{4i} \right\}
$$
Note that by \cite[Theorem 3.8]{AHJOW}, $\calG_1'$ and $\calG_2$ are adjacent on $\R$ (note that both $\calG_1'$ and $\calG_2$ are of base $16$). This suggests that in the adjacent pair $\{\calG_1, \calG_2\}$, only the cubes in $\left(\calG_1\right)_{4i}, i \in \Z$ (that is, cubes only every four levels) contribute to the adjacency, while the cubes from other generations are redundant.

\end{exa}

Let us make the above phenomenon in a quantitative way. To do this, we first introduce the following auxiliary function: for any $n, n' \in \N$ with $n, n' \ge 2$, $\phi_{n; n'}: \N \to \N$ is given by 
$$
\phi_{n; n'}(j):=\left\lfloor \frac{j\log n}{\log n'} \right\rfloor. 
$$
Note that $\phi_{n, n}(j)=j$ for all $j \in \N$ and $n \ge 2$. 

This auxiliary function allows us to extend the concept of ``\emph{far}" considered in \cite{AH} and \cite{AHJOW}, which is the first ingredient that we need for our main result. Such a generalization is two-fold: first of all, we are able to define ``\emph{far number with respect to a finite collection of integers}" (see, Definition \ref {20200910defn01}); second, we also describe the ``\emph{far pair of integer-valued functions with respect to a finite collection of integers}" (see, Definition \ref{20200914defn01}). Now let us turn to some details. 

\begin{defn} \label{20200910defn01}
Let $\calN:=\{n_1, \dots, n_L\}$ be a collection of positive integers where each $n_\ell \ge 2, \ell=1, \dots, L$. Given any $\del \in \R$ and $n, n' \in \N$ with $n, n' \ge 2$, we say $\del$ is  \emph{a $(n, n')$-far number with respect to $\calN$} if there exists $C>0$ such that for any $n_\ell \in \calN$, there holds that
\begin{equation} \label{general far}
  \left| \del-\frac{k_1}{n^{\phi_{n_{\ell}; n}(m)}}-\frac{k_2}{\left(n'\right)^{\phi_{n_{\ell}; n'}(m)}} \right| \ge \frac{C}{n_{\ell}^m}, \quad \forall m \ge 0, k_1, k_2\in \Z,
\end{equation}
where $C$ only depends on $n$, $n'$, $\del$, $\calN$, $L$ and any dimension constants, but is independent of $m$, $k_1$ and $k_2$. 
\end{defn}

\begin{rem}
\begin{enumerate}
    \item [(1).] The concept of far numbers with respect to a set will be used to deal with the \emph{small scale case} in most of our applications later, and we will only consider the case when $n, n' \in \calN$ and $L=d+1$, which is the optimal number of general dyadic systems that needed to guarantee the adjacency in $\R^d$. Therefore, the constant $C$ in \eqref{general far} later will only depend on $\del, \calN$ and any dimensional constants;
    
    \medskip
    
    \item [(2).] When $\calN=\{n_0\}$ and $n=n'=n_0$, Definition \ref{20200910defn01} coincides with the classical definition of $n_0$-far number, which was considered in \cite{AHJOW} and \cite{AH}. The concept of \emph{far numbers}, where $\calN=\{2\}$,  was introduced by Mei \cite{TM} to prove that the one-parameter space BMO($\T$), which is the space of bounded mean oscillation on the torus $\T$, can be written as the intersection of two dyadic product BMO($\T$) spaces, with equivalent norms. In 2013, Li, Pipher and Ward \cite{LPW} generalized Mei's result to multi-parameter case and a vast class of function spaces via a more careful study of far numbers. For a systematic study of far numbers, we refer the interested reader to \cite{AHJOW} for more details. 
\end{enumerate}
\end{rem}

Next, we define the ``far pairs of integer-valued functions with respect to a finite set". 

\begin{defn} \label{20200914defn01}
Let $\calN$ and $L$ be defined as above. Given any integer valued functions $\calL, \calL': \N \to \N$ and $n, n' \in \N$ with $n, n' \ge 2$, we say $(\calL, \calL')$ is \emph{a $(n, n')$-far pair of integer-valued functions with respect to $\calN$} if there exists a $C'>0$ and $J \in \N$  sufficiently large, such that for any $n_\ell \in \calN$, any $j \ge J$ and any $k_3, k_4 \in \Z$, there holds that
\begin{equation} \label{20200914eq03}
\left| \calL\left( \phi_{n_\ell, n}(j) \right)+k_3 n^{\phi_{n_\ell; n}(j)}-\calL'\left( \phi_{n_\ell, n'}(j) \right)-k_4 \left(n' \right)^{\phi_{n_\ell; n'}(j)} \right| \ge C'n_\ell^j,
\end{equation}
where $C'$ only depends on $n, n', \calL, \calL', \calN, L, J$ and any dimensional constants, but independent of $j$, $k_3$ and $k_4$. 
\end{defn}

\begin{rem}
The concept of far pairs of integer-valued functions with respect to a finite set will be used to deal with the \emph{large scale case} in our application later, where again, we will only consider the case when $n, n' \in \calN$ and $L=d+1$. 
\end{rem}

The second concept that we need is the \emph{representation of general dyadic systems}. The setting is as follows.  
\begin{enumerate}
    \item [(1).] $\del \in \R^d$, which can be interpreted as the ``initial point" to build the grid; 
    
    \item [(2).] $n \in \N$ with $n \ge 2$, which is the base of the grid;
    
    \item [(3).] An infinite matrix
 \begin{equation} \label{20200910eq02}
    \Vec{\textbf{a}}:=\left\{\Vec{a}_0, \dots, \Vec{a}_j, \dots \right\},
    \end{equation} 
    where $\Vec{a}_j \in \{0,1,\dots , n-1\}^d, j \ge 1$;
    
    \item [(4).] The \emph{location function} associated to $\Vec{\textbf{a}}$: 
    $$
    \calL_{\bfVa}: \N \longmapsto \Z^d, 
    $$
    which is defined by 
    $$
\calL_{\bfVa}(j):=
\begin{cases}
\sum\limits_{i=0}^{j-1} n^i\Vec{a}_i, \hfill \quad \quad \quad j \ge 1;\\
\\
\Vec{0}, \hfill  \quad \quad \quad j=0. 
\end{cases}
$$
\end{enumerate}

\begin{defn} \label{repsgrids}
Let $\del \in \R^d$, $n \ge 2$ be an integer, $\bfVa$ and $\calL_{\bfVa}$ be defined as above. Let $\calG(n, \del, \calL_{\bfVa})$ be the collection of the following cubes:
\begin{enumerate}
    \item [(1).] For $m \ge 0$, the $m$-th generation of $\calG(\del, \calL_{\bfVa})$ is defined as
\begin{eqnarray*}
&&\calG(n, \del)_m:=\Bigg\{ \left[ (\del)_1+\frac{k_1}{n^m}, (\del)_1+\frac{k_1+1}{n^m} \right) \times \dots \\
&& \quad \quad \quad  \quad \quad \quad \quad \quad  \quad \quad \quad \times \left[ (\del)_d+\frac{k_d}{n^m}, (\del)_d+\frac{k_d+1}{n^m} \right) \bigg | (k_1, \dots, k_d) \in \Z^d \Bigg\}, 
\end{eqnarray*}
where here and in the sequel, we use $(\del)_s$ to denote the $s$-th component of a vector $\del \in \R^d$.

Note that all the positive generations (that is, the collection of cubes with sidelength less or equal to $1$) are uniquely determined by the initial point $\del$, and hence the location function $\calL_{\bfVa}$ does not make any contribution for positive generations;  

\medskip

\item [(2).] For $m<0$, the $m$-th generation is defined as 
\begin{eqnarray*}
&& \calG(n, \del, \calL_{\bfVa})_m:=\Bigg\{ \left[ (\del)_1+\left[ \calL_{\bfVa}(-m) \right]_1+\frac{k_1}{n^m}, (\del)_1+\left[ \calL_{\bfVa}(-m) \right]_1+\frac{k_1+1}{n^m} \right) \times \dots  \nonumber \\
&& \quad \quad  \times \left[ (\del)_d+\left[ \calL_{\bfVa}(-m) \right]_d+\frac{k_d}{n^m}, (\del)_d+\left[ \calL_{\bfVa}(-m) \right]_d+\frac{k_d+1}{n^m} \right)  \bigg | (k_1, \dots, k_d) \in \Z^d \Bigg\}. 
\end{eqnarray*}
\end{enumerate}

To this end, for each $m \in \N$, we denote $b \left[\calG\left(n, \del, \calL_{\bfVa} \right)_m \right]$ as the collection of all the boundaries of the cubes in $\calG\left(n, \del, \calL_{\bfVa} \right)_m$. 
\end{defn}

\begin{rem}
 \begin{enumerate}
     \item [1.] The term $\del+\calL_{\bfVa}(-m)$ in Definition \ref{repsgrids}  can be interpreted as the location of $\del$ after choosing $n$-adic parents (with respect to the $0$-th generation) $m$ times;
     
     \medskip
     
     \item [2.] $\calG(n, \del, \calL_{\bfVa})$ is a $n$-adic grid; on the other hand, for any $n$-adic grid $\calG$, it can be represented as $\calG(n, \del, \calL_{\bfVa})$, for some $\del \in \R^n$ and infinite matrix $\bfVa$ defined in \eqref{20200910eq02} (see, \cite[Proposition 3.2 and Proposition 3.3]{AH}). Moreover, although the representation of a $n$-adic grid in general is not unique, they are essentially the ``same" from the view of adjacency (see, \cite[Theorem 3.14]{AHJOW} for both the real line case and \cite[Corollary 2.5]{AH} for the higher dimensional case).
 \end{enumerate}
\end{rem}

We are ready to state our main result, which generalizes \cite[Theorem 1.5]{AH}. 

\begin{thm} \label{mainresult01}
Let $\calG_i:=\calG(n_i, \del_i, \calL_{\bfVa_i})$, $i=1, \dots, d+1$ be a collection of general dyadic grids, where $d, n_i, \del_i$ and $\bfVa_i$ are defined as above. Let further, $\calN:=\{n_1, \dots, n_{d+1}\}$. Then $\calG_1, \dots, \calG_{d+1}$ are adjacent if and only if the following conditions hold: 
\begin{enumerate}
\item [(1).] For any $\ell_1, \ell_2 \in \{1, \dots, d+1\}$ where $\ell_1 \neq \ell_2$, and $s \in \{1, \dots, d\}$, 
$$
\left(\del_{\ell_1} \right)_s-\left(\del_{\ell_2} \right)_s
$$ 
is a $(n_{\ell_1}, n_{\ell_2})$-far number with respect to $\calN$; 

\medskip

\item [(2).] For any $\ell_3, \ell_4 \in \{1, \dots, d+1\}, \ell_3 \neq \ell_4$, and $s \in \{1, \dots, d\}$, the pairs of integer valued function 
$$
\left(\left[\calL_{\bfVa_{\ell_3}}(\cdot)\right]_s, \left[\calL_{\bfVa_{\ell_4}}(\cdot)\right]_s\right)
$$
is a $(n_{\ell_3}, n_{\ell_4})$-far pair of integer-valued functions with respect to $\calN$. 
\end{enumerate}
\end{thm}

\begin{rem}
 \begin{enumerate}
     \item [1.] Theorem \ref{mainresult01} is sharp, in the sense that the number of the general dyadic systems needed to guarantee the adjacency cannot be reduced; 
     
     \medskip
     
     \item [2.] Let us include some motivation for the auxiliary function $\phi$. As we have pointed out earlier, the new feature for the general dyadic systems with different bases is that adjacency is given by cubes with different sidelengths from different grids. For example, in \cite{AH}, the term we have for the second condition in our main result is 
     $$
     \left|\frac{ \left[\calL_{\bfVa_{k_1}}(j)\right]_{s}-\left[\calL_{\bfVa_{k_2}}(j)\right]_{s}}{n^j}\right|.
     $$
     However, if the general dyadic grids are allowed to have different bases, it is no longer correct to compare the location functions at the same ``level" $j$ (that is, $(-j)$-th generation), for instance, consider a term like
     \begin{equation} \label{20200910eq03}
     \left|\frac{ \left[\calL_{\bfVa_{k_1}}(j)\right]_{s}-\left[\calL_{\bfVa_{k_2}}(j)\right]_{s}}{n_{k_1}^j}\right|.
     \end{equation}
     Heuristically, let us assume $n_{k_1} \ge n_{k_2}$, and $\bfVa_{k_1}$ is an infinite matrix such that $$
     \left[\calL_{\bfVa_{k_1}}(j) \right]_s \sim n_{k_1}^j
     $$ 
     when $j$ is large. However, $0 \le \left[\calL_{\bfVa_{k_2}}(j) \right]_s <n_{k_2}^j$, which is negligible compared to the term $\left[\calL_{\bfVa_{k_1}}(j) \right]_s$. In other words, \eqref{20200910eq03} suggests that the adjacency for the large scale only depends on grids with a larger base, which is not correct since we can always do the ``changing base trick" as in \eqref{20200910eq04} to make the base as large as we want. A similar observation suggests that for the \emph{small scale case}, one has to extend the definition of ``far with respect to a number" to ``far with respect to a set" (see, Definition \ref{20200910defn01}). 
     \medskip
     
     These phenomena suggest that when the bases are different, it is more reasonable to explore how the cubes from different grids with comparable size interact with each other, rather than the from the same generations. This motivates us to introduce $\phi$ to quantify such a phenomenon. 
     
     To this end, we make a comment that another possible approach to study the geometry underlying Theorem \ref{mainresult01} is to consider the \emph{fundamental structures of $d+1$ general dyadic dyadic systems with different bases}, which were introduced in \cite{AH} to study the adjacency of general dyadic grids with the same base. It is not hard to see that when the bases are different, these structures make sense if and only if the cubes used to build these structures from different grids are of comparable sizes. 
\end{enumerate}
\end{rem}

The following corollary is straightforward from the main result Theorem \ref{mainresult01} (see \cite{AH} for a more detailed explanation).

\begin{cor} \label{20200915cor01}
Let $\calG_1, \dots, \calG_{d+1}$ be defined as in Theorem \ref{mainresult01}. $\calG_1, \dots, \calG_{d+1}$ is adjacent in $\R^d$ if and only if the projection of any two of them onto any coordinate axis is adjacent in $\R$. 
\end{cor}

A second application of our main result Theorem \ref{mainresult01} is that the typical examples provided by the ``changing bases trick" in Example \ref{20200922exa01} are actually the only possible cases for adjacent general dyadic systems with different bases. More precisely, we have the following result. 

\begin{thm} \label{mainresult02}
Let $\calG_1, \dots, \calG_{d+1}$ be adjacent on $\R^d$. Then there exists an integer $\mathfrak n \ge 2$, and $s_i \in \N, s_i \ge 1, i=1, \dots, d+1$, such that
$$
n_i=\mathfrak n^{s_i}, \quad i=1, \dots, d+1. 
$$
\end{thm}

\begin{rem}
 Note that if such an $\mathfrak n$ does not exist, then this means that $\frac{\log n_i}{\log n_j}$ is irrational for some $i \neq j$, whereas if $\mathfrak n$ does exist, then  $\frac{\log n_i}{\log n_j}$ is always rational.  This is related to numbers normal to different bases, that is, $\frac{\log n_1}{\log n_2}$ is rational if and only if every number that is normal to base $n_1$ is also normal to base $n_2$ \cite{Schmidt}.  Also, via work of Wu \cite{Wu}, $\frac{\log n_1}{\log n_2}$ being rational means that the null sets for the $n_1$-adic doubling measures and $n_2$-adic doubling measures are equal.  This in turn loosely relates to our recent work \cite{AH2}.
\end{rem}
The structure of the paper is as follows: Section 2 is devoted to prove the main result Theorem \ref{mainresult01}. Moreover, we also show Theorem \ref{mainresult01} is independent of the representation. While in Section 3, we prove Theorem \ref{mainresult02} and connect the discoveries therein with other recent work. 

\bigskip
\section{Proof of Theorem \ref{mainresult01}}

In this section, we prove our main result Theorem \ref{mainresult01}.

\subsection{Necessity.} Suppose $\calG_1=\calG(n_1, \del_1, \calL_{\bfVa_1}), \dots, \calG_{d+1}=\calG(n_{d+1}, \del_{d+1}, \calL_{\bfVa_{d+1}})$ are adjacent. We prove the result by contradiction. 

\medskip

Assume condition (1) fails. This means we can take some $\ell_1, \ell_2 \in \{1, \dots, d+1\}$ with $\ell_1 \neq \ell_2$ and $s \in \{1, \dots, d\}$, such that for each $N_1 \ge 1$, there exists some $m \ge 0$, $K_1, K_2 \in \Z$ and $n_\ell \in \calN$, such that
\begin{equation} \label{20200910eq11}
\left| \left(\del_{\ell_1} \right)_s-\left(\del_{\ell_2} \right)_s- \frac{K_1}{n_{\ell_1}^{\phi_{n_\ell; n_{\ell_1}}(m)}}-\frac{K_2}{n_{\ell_2}^{\phi_{n_\ell; n_{\ell_2}}(m)}}\right|< \frac{1}{N_1 n_{\ell}^{m}}. 
\end{equation}
This implies the distance between the hyperplane $\left\{(x)_s=(\del_{\ell_1})_s-\frac{K_1}{n_{\ell_1}^{\phi_{n_\ell; n_{\ell_1}}(m)}}\right\}$ and the hyperplane $\left\{(x)_s=(\del_{\ell_2})_s+\frac{K_2}{n_{\ell_2}^{\phi_{n_\ell; n_{\ell_2}}(m)}}\right\}$ is less than $\frac{1}{N_1 n_{\ell}^{m}}$. 

Observe that
$$
\left\{(x)_s=(\del_{\ell_1})_s-\frac{K_1}{n_{\ell_1}^{\phi_{n_\ell; n_{\ell_1}}(m)}}\right\} \subset b\left[ \left(\calG_{\ell_1} \right)_{\phi_{n_\ell; n_{\ell_1}}(m)} \right]
$$
and 
$$
\left\{(x)_s=(\del_{\ell_2})_s+\frac{K_2}{n_{\ell_2}^{\phi_{n_\ell; n_{\ell_2}}(m)}}\right\} \subset b\left[ \left(\calG_{\ell_2}\right)_{\phi_{n_\ell; n_{\ell_2}}(m)} \right]. 
$$
Therefore, the estimate \eqref{20200910eq11} suggests that we can take two sufficiently close points, with the first one located on $b\left[ \left(\calG_{\ell_1} \right)_{\phi_{n_\ell; n_{\ell_1}}(m)} \right]$, and the second one on  $b\left[ \left(\calG_{\ell_2}\right)_{\phi_{n_\ell; n_{\ell_2}}(m)} \right]$, respectively. More precisely, we may assume $s=\ell=\ell_1=1$ and $\ell_2=2$ for simplicity. We consider the points $p_1$ and $p_2$, which are given by intersection of $d$ non-parallel hyperplanes as follows:
$$
p_1:=\left\{(x)_1=(\del_1)_1-\frac{K_1}{n_1^{m}}\right\} \cap \{(x)_2=(\del_3)_2 \} \cap \cdots \cap \{(x)_d=(\del_{d+1})_d \}
$$
and
$$
p_2:=\left\{(x)_1=(\del_2)_1+\frac{K_2}{n_2^{\phi_{n_1; n_2}(m)}}\right\} \cap \{(x)_2=(\del_3)_2 \} \cap \cdots \cap \{(x)_d=(\del_{d+1})_d \} \}.
$$
Note that $p_1$ and $p_2$ enjoy the following properties:
\begin{enumerate}
    \item [(a).] $p_1 \in b\left[ \left(\calG_1 \right)_m \right] \cap \left(\bigcap\limits_{t=3}^{d+1} b\left[ \left(\calG_t\right)_0 \right]\right)$ and  $p_2 \in b\left[\left(\calG_2 \right)_{\phi_{n_1; n_2}(m)} \right] \cap \bigcap\limits_{t=3}^{d+1} b\left[ \left(\calG_t \right)_0 \right]$;
    
    \medskip
    
    \item [(b).] $\textrm{dist}(p_1, p_2)< \frac{1}{N_1 n_1^{m}}$. 
\end{enumerate}

Note that the second property above allows us to choose an open cube $Q$ of sidelength $\frac{1}{N_1 n_1^{m}}$ containing both $p_1$ and $p_2$ as interior points; while the first property asserts that if there is a dyadic cube $D \in \calG_{\ell}, \ell \in \{1, \dots, d+1\}$ covering $Q$, then the sidelength of $D$ is at least $\frac{1}{2n_1^m}$. Indeed, if $D \in \calG_1$, then since $D \cap b\left[ \left(\calG_1 \right)_m \right] \neq \emptyset$, $D$ has to belong to $\left(\calG_1 \right)_{m'}$ for some $m'>m$, this suggests $\ell(D)>\frac{1}{n_1^m}$; if $D \in \calG_2$, then similarly we have
\begin{eqnarray} \label{20200915eq01}
\ell(D) %
&>& \frac{1}{n_2^{{\phi_{n_1; n_2}(m)}}} = \frac{1}{ e^{\log n_2 \cdot \left\lfloor \frac{m\log n_1}{\log n_2} \right\rfloor}} \nonumber \\ 
& \ge & \frac{1}{e^{\log n_2 \cdot \frac{m \log n_1}{\log n_2}}}= \frac{1}{n_1^m}. 
\end{eqnarray}
Finally, if $D \in \calG_i$, $i \in \{3, \dots d+1\}$, a similar argument suggests that $\ell(D) \ge 1$. Hence
$$
\ell(D)>\frac{N_1}{2} \cdot \ell(Q). 
$$
This will contradict adjacency if we choose $N_1$ sufficiently large. 

\medskip

Next, we assume condition (2) fails.  This means that there is some $\ell_3, \ell_4 \in \{1, \dots, d+1\}$ with $\ell_3 \neq \ell_4$ and $s \in \{1, \dots, d\}$, such that for each $N_2 \ge 1$, there exists some $n_\ell \in \calN$, $j$ sufficiently large and $K_3, K_4 \in \Z$, such that
\begin{equation} \label{20200911eq02}
\left| \left[\calL_{\bfVa_{\ell_3}}\left( \phi_{n_\ell; n_{\ell_3}}(j) \right) \right]_s+K_3 n_{\ell_3}^{\phi_{n_\ell; n_{\ell_3}}(j)}- \left[\calL_{\bfVa_{\ell_4}}\left( \phi_{n_\ell; n_{\ell_4}}(j) \right) \right]_s-K_4 n_{\ell_4}^{\phi_{n_\ell; n_{\ell_4}}(j)}\right|<\frac{n_\ell^j}{N_2}. 
\end{equation}
Without loss of generality, we consider the case when $s=\ell=\ell_3=1$ and $\ell_4=2$. Therefore, \eqref{20200911eq02} can be simplified as 
\begin{equation} \label{20200914eq04}
\left| \left[\calL_{\bfVa_1}(j) \right]_1+K_3n_1^j- \left[\calL_{\bfVa_2} \left( \phi_{n_1; n_2}(j) \right) \right]_1-K_4n_2^{\phi_{n_1; n_2}(j)} \right|<\frac{n_1^j}{N_2}. 
\end{equation} 
This gives
$$
\left| \left[\del_1+\calL_{\bfVa_1}(j) \right]_1+K_3n_1^j- \left[\del_2+\calL_{\bfVa_2} \left( \phi_{n_1; n_2}(j) \right) \right]_1-K_4n_2^{\phi_{n_1; n_2}(j)} \right|<\frac{2n_1^j}{N_2},
$$
since we can pick $j$ is sufficiently large. To this end, let us rewrite the above estimate as follows
\begin{equation} \label{20200914eq05}
\left| \left[\del_1+\calL_{\bfVa_1}(j)+K_3n_1^j \vec{e}_1 \right]_1- \left[\del_2+\calL_{\bfVa_2} \left( \phi_{n_1; n_2}(j) \right) +K_4n_2^{\phi_{n_1; n_2}(j)} \vec{e}_1 \right]_1 \right|<\frac{2n_1^j}{N_2},
\end{equation} 
where for $s \in \{1, \dots, d\}$, $\vec{e}_s$ is the standard unit vector in $\R^d$ with $s$-th entry being $1$.

Similarly as \eqref{20200910eq11}, the estimate \eqref{20200914eq05} can also be interpreted geometrically. Indeed, let
\begin{eqnarray*}
q_1%
&:=& \left\{(x)_1=\left[\del_1+\calL_{\bfVa_1}(j)+K_3n_1^j \vec{e}_1 \right]_1\right\} \cap \\
&& \quad \quad \quad \quad  \quad \quad \bigcap_{t=2}^d \left\{(x)_t=\left[\del_{t+1}+\calL_{\bfVa_{t+1}} \left( \phi_{n_1; n_{t+1}}(j) \right) \right]_t \right\}  
\end{eqnarray*}
and 
\begin{eqnarray*}
q_2%
&:=& \left\{(x)_1=  \left[\del_2+\calL_{\bfVa_2} \left( \phi_{n_1; n_2}(j) \right) +K_4n_2^{\phi_{n_1; n_2}(j)} \vec{e}_1 \right]_1 \right\} \cap \\
&& \quad \quad \quad \quad  \quad \quad \bigcap_{t=2}^d \left\{(x)_t=\left[\del_{t+1}+\calL_{\bfVa_{t+1}} \left( \phi_{n_1; n_{t+1}}(j) \right) \right]_t \right\}.  
\end{eqnarray*}
which satisfies the following properties:
\begin{enumerate}
    \item [(a).] $q_1 \in b\left[ \left(\calG_1 \right)_{-j} \right] \cap \left(\bigcap\limits_{k=3}^{d+1} b\left[ \left(\calG_k\right)_{-\phi_{n_1; n_k}(j)} \right] \right)$ and $q_2 \in \bigcap\limits_{k=2}^{d+1} b\left[ \left(\calG_k\right)_{-\phi_{n_1; n_k}(j)} \right]$;
    
    \medskip
    \item [(b).] $\textrm{dist}(q_1, q_2)<\frac{2n_1^j}{N_2}$.
\end{enumerate}
Note that the reason for us to consider the negative generations in property (a) above is that the cubes underlying the condition (2) is of side length greater than $1$ (namely, the \emph{large scale case}). 

To derive the desired contradiction, we again start from the second condition above by taking a cube $Q$ containing both $q_1$ and $q_2$, with $\ell(Q)=\frac{2n_1^j}{N_2}$. Now we let $D \in \calG_i$ for some $i \in \{1, \dots, d+1\}$ which contains $Q$. However, since 
$$
Q \cap b\left[ \left(\calG_i\right)_{-\phi_{n_1; n_i}(j_1)} \right] \neq \emptyset, 
$$
(note that $\phi_{n; n}(j)=j$ for $j$ large), it follows that 
\begin{eqnarray} \label{20200911eq11}
\ell(D)%
&>& n_i^{\phi_{n_1; n_i}(j_1)}=\exp \left( \log n_i \cdot \left\lfloor \frac{j_1\log n_1}{\log n_i} \right\rfloor \right) \nonumber \\
&\ge& \exp \left( \log n_i\cdot \left(\frac{j_1 \log n_1}{\log n_i}-1 \right) \right) = \frac{n_1^{j_1}}{n_i}. 
\end{eqnarray}
Therefore, we have 
$$
\ell(D)> \frac{N_2\ell(Q)}{2n_i},
$$
which is again a contradiction if we choose $N_2$ sufficiently large. 

\medskip

The proof for the necessity is complete.

\medskip

\subsection{Sufficiency.}

Suppose conditions (1) and (2) hold, that is,
\begin{enumerate}
    \item [(1).] For any $\ell_1, \ell_2 \in \{1, \dots, d+1\}$ where $\ell_1 \neq \ell_2$, and $s \in \{1, \dots, d\}$,  there exists some constant $C(\ell_1, \ell_2, s)>0$, such that for any $m \ge 0$, any $n_\ell \in \calN$ and $k_1, k_2 \in \Z$, there holds
\begin{equation} \label{20200911eq04}
\left| \left(\del_{\ell_1} \right)_s-\left(\del_{\ell_2} \right)_s- \frac{k_1}{n_{\ell_1}^{\phi_{n_\ell; n_{\ell_1}}(m)}}-\frac{k_2}{n_{\ell_2}^{\phi_{n_\ell; n_{\ell_2}}(m)}}\right| \ge \frac{C(\ell_1, \ell_2, s)}{n_{\ell}^m}.
\end{equation}

\medskip

\item [(2).] For any $\ell_3, \ell_4 \in \{1, \dots, d+1\}, \ell_3 \neq \ell_4$, and $s \in \{1, \dots, d\}$, there exists a $C'(\ell_3, \ell_4, s)>0$ and $J(\ell_3, \ell_4, s) \in \N$ sufficiently large, such that for any $n_\ell \in \calN$, $j>J(\ell_3, \ell_4, s)$ and any $k_3, k_4 \in \Z$, there holds that
\begin{eqnarray} \label{20200911eq06} 
&&\bigg| \left[ \calL_{\bfVa_{\ell_3}} \left( \phi_{n_\ell, n_{\ell_3}}(j) \right) \right]_s+k_3n_{\ell_3}^{ \phi_{n_\ell, n_{\ell_3}}(j)} \nonumber \\
&& \quad \quad \quad \quad \quad \quad \quad  -\left[ \calL_{\bfVa_{\ell_4}} \left( \phi_{n_\ell, n_{\ell_4}}(j) \right) \right]_s-k_4n_{\ell_4}^{\phi_{n_\ell, n_{\ell_4}}(j)}  \bigg| \ge C'(\ell_3, \ell_4, s) \cdot n_\ell^j. 
\end{eqnarray}
\end{enumerate}
Denote 
$$
C_1:=\min_{\substack{1 \le \ell_1 \neq  \ell_2 \le d+1 \\ 1 \le s \le d}} C(\ell_1, \ell_2, s),  \quad
C'_1:=\min_{\substack{1 \le \ell_3 \neq \ell_4 \le d+1 \\ 1 \le s \le d}} C'(\ell_3,\ell_4, s). 
$$
and
\begin{equation} \label{20200915eq05}
J:=\max_{\substack{1 \le \ell_3 \neq \ell_4 \le d+1 \\ 1 \le s \le d}} J(\ell_3, \ell_4, s). 
\end{equation}

It is clear that $C_1, C_1'>0$.

\medskip

Recall the goal is to show the collection $\calG_1, \dots, \calG_{d+1}$ is adjacent on $\R^d$. Take some $C>0$ be a constant such that
$$
0<C<\min \left\{  C_1, \frac{C_1'}{2} \right\}.
$$
Let $Q$ be any cube in $\R^d$ and let $m_0 \in \Z$, such that
\begin{equation} \label{20200911eq10}
\frac{C}{n_1^{m_0+1}} \le \ell(Q)<\frac{C}{n_1^{m_0}}. 
\end{equation} 
Let us consider several cases. 

\medskip

\textit{Case I: $m_0>0$.} We have the following claim: there exists some $\ell \in \{1, \dots, d+1\}$, such that
$$
Q \in \left( \calG_\ell \right)_{\phi_{n_1; n_\ell}(m_0)}. 
$$

\medskip

\textit{Proof of the claim:} We prove the claim by contradiction. If $Q$ is not contained in any cubes from $\left(\calG_\ell \right)_{\phi_{n_1; n_\ell}(m_0)}$ for any $\ell \in \{1, \dots, d+1\}$, then for each $\ell \in \{1, \dots, d+1\}$, we can find a $s_\ell \in \{1, \dots, d\}$, such that
$$
P_{s_\ell}(Q) \cap P_{s_\ell} \left( \calV_{\ell, \phi_{n_1; n_\ell}(m_0)} \right) \neq \emptyset, 
$$
where $P_s$ is the orthogonal projection onto the $s$-th coordinate axis in $\R^d$ for $s \in \{1, \dots, d\}$, and for $\ell \in \{1, \dots, d+1\}$ and $m>0$, 
$$
\calV_{\ell, m}:=\left\{ \del_\ell+\frac{\vec{v}}{n_\ell^m}: \vec{v} \in \Z^d \right\}
$$
is the collection of all the vertices of the cubes in $\left(\calG_\ell \right)_m$. By pigeonholing, there exists some $\ell_1, \ell_2 \in \{1, \dots, d+1\}$ with $\ell_1 \neq \ell_2$, but $s_*:=s_{\ell_1}=s_{\ell_2}$, such that
$$
P_{s_*}(Q) \cap P_{s_*} \left(\calV_{\ell_1, \phi_{n_1; n_{\ell_1}(m_0)}} \right), \quad P_{s_*}(Q) \cap P_{s_*} \left(\calV_{\ell_2, \phi_{n_1; n_{\ell_2}(m_0)}} \right) \neq \emptyset,
$$
which implies that there exists some $K_1, K_2 \in \Z$, such that
$$
\left| \left(\del_{\ell_1}\right)_{s_*}+\frac{K_1}{n_{\ell_1}^{\phi_{n_1; n_{\ell_1}}(m_0)}} -\left(\del_{\ell_2} \right)_{s_*}-\frac{K_2}{n_{\ell_2}^{\phi_{n_1; n_{\ell_2}}(m_0)}} \right| \le \ell(Q)<\frac{C}{n_1^{m_0}}< \frac{C(\ell_1, \ell_2, s_*)}{n_1^{m_0}}, 
$$
which contradicts \eqref{20200911eq04}.

\medskip

\textit{Case II: $m_0 \le -J$, where $J$ is defined in \eqref{20200915eq05}.} In this case, our goal is to show that
$$
Q \in \left(\calG_k \right)_{-\phi_{n_1; n_\ell}(-m_0)}
$$
for some $\ell \in \{1, \dots, d+1\}$. We prove it by contradiction again. Following the argument in Case I above, we see that there exists some $\ell_3, \ell_4 \in \{1, \dots, d+1\}$ with $\ell_3 \neq \ell_4$ and $s^* \in \{1, \dots, d\}$, such that 
$$
P_{s^*} (Q) \cap P_{s^*} (\calV_{\ell_3, -\phi_{n_1; n_{\ell_3}}(-m_0)}), \quad P_{s^*} (Q) \cap P_{s^*} (\calV_{\ell_4, -\phi_{n_1; n_{\ell_4}}(-m_0)}) \neq \emptyset.
$$

This implies there exists some $K_3, K_4 \in \Z$, such that
\begin{eqnarray*}
&& \bigg| \left[\del_{\ell_3}+\calL_{\bfVa_{\ell_3}}\left(-\phi_{n_1; n_{\ell_3}}(-m_0) \right) \right]_{s^*}+K_3n_{\ell_3}^{\phi_{n_1; n_{\ell_3}}(-m_0)}- \\
&& \quad \quad \quad \quad \quad  \quad \quad \quad \left[\del_{\ell_4}+\calL_{\bfVa_{\ell_4}}\left(-\phi_{n_1; n_{\ell_3}}(-m_0) \right) \right]_{s^*}-K_4n_{\ell_4}^{\phi_{n_1; n_{\ell_4}}(-m_0)} \bigg|<\frac{C}{n_1^{m_0}}. 
\end{eqnarray*}
Note that since we can always choose $J$ sufficiently large, we can indeed reduce the above estimate to
\begin{eqnarray*}
&& \bigg| \left[\calL_{\bfVa_{\ell_3}}\left(-\phi_{n_1; n_{\ell_3}}(-m_0) \right) \right]_{s^*}+K_3n_{\ell_3}^{\phi_{n_1; n_{\ell_3}}(-m_0)}- \\
&& \quad \quad \quad \quad \quad  \quad \quad \quad \left[\calL_{\bfVa_{\ell_4}}\left(-\phi_{n_1; n_{\ell_3}}(-m_0) \right) \right]_{s^*}-K_4n_{\ell_4}^{\phi_{n_1; n_{\ell_4}}(-m_0)} \bigg|<\frac{2C}{n_1^{m_0}}. 
\end{eqnarray*}
This gives the desired contradiction to \eqref{20200911eq06}, since $C<\frac{C_1'}{2} \le \frac{C_1'(\ell_3, \ell_4, s^*)}{2}$.

\medskip

\textit{Case III: $-J<m_0 \le 0$.} Indeed, we can ``pass" the third case to the second case, by taking a cube $Q'$ containing $Q$ with the sidelength is $n_1^J$. Applying the second case to $Q'$, we find that there exists some $D \in \calG_k$ for some $k \in \{1, \dots, d+1\}$, such that $Q' \subset D$ and $\ell(D) \le C_4 \ell(Q')$, which clearly implies 
\begin{enumerate}
    \item [(1).] $Q \subset D$;
    \item [(2).] $\ell(D) \le C_4 n_1^J \ell(Q)$. 
\end{enumerate}

\medskip

The proof is complete. \hfill $\square$

\bigskip

Finally, we make a remark that Theorem \ref{mainresult01} is \emph{independent} of the choice of the representation. Note that by Corollary \ref{20200915cor01}, it suffices to consider the case when $d=1$. Let $\calG_1:=\calG\left(n_1, \del_1, \calL_{\textbf{a}} \right)$ and $\calG_2:=\calG \left(n_2, \del_2, \calL_{\textbf{b}} \right)$, where $\textbf{a}=(a_1, \dots, a_i, \dots)$ is an infinite sequence of integers where $a_i \in \{0, \dots, n_1-1\}$ and $\textbf{b}$ can be defined similarly with $n_1$ being replaced by $n_2$. Then Theorem \ref{mainresult01} asserts that $\calG_1$ and $\calG_2$ are adjacent on $\R$ if and only if 
\begin{enumerate}
    \item [(1).] there exists some $C>0$, such that for any $m \ge 0$ and $k_1, k_2 \in \Z$, there holds
 \begin{equation} \label{20200922eq01}
    \left| \del_1-\del_2-\frac{k_1}{n_1^{\phi_{n_\ell; n_1}(m)}}-\frac{k_2}{n_2^{\phi_{n_\ell; n_2}(m)}} \right| \ge \frac{C}{n_\ell^m}, \quad \ell=1, 2;
    \end{equation} 
    
    \medskip
    
    \item [(2).] there exists some $C'>0$ and $J \in \N$ sufficiently large, such that for any $j>J$ and any $k_3,  k_4 \in \Z$, there holds
    $$
    \left| \calL_{\textbf{a}}\left(\phi_{n_\ell; n_1}(j) \right)+k_3n_1^{\phi_{n_\ell; n_1}(j)}-\calL_{\textbf{b}}\left(\phi_{n_\ell; n_2}(j) \right)-k_4n_2^{\phi_{n_\ell; n_2}(j)} \right| \ge C'n_\ell^j, \quad  \ell=1, 2,
    $$
    in other words, 
    \begin{eqnarray} \label{20200918eq01}
    && \big| \del_1+\calL_{\textbf{a}}\left(\phi_{n_\ell; n_1}(j) \right)+k_3n_1^{\phi_{n_\ell; n_1}(j)} \nonumber \\
    && \quad \quad \quad \quad \quad \quad \quad -\del_2-\calL_{\textbf{b}}\left(\phi_{n_\ell; n_2}(j) \right)-k_4n_2^{\phi_{n_\ell; n_2}(j)} \big| \ge \frac{C'n_\ell^j}{2}, \quad  \ell=1, 2.
    \end{eqnarray}
\end{enumerate}

The goal now is to show that the constants $C$ and $C'$ defined in above, respectively, are independent of the choice of the representation. Let $\calG(n_1, \del_1', \calL_{\textbf{a}'})$ and $\calG(n_2, \del_2', \calL_{\textbf{b}'})$ be some other representations of $\calG_1$ and $\calG_2$, respectively. Note the following two facts:
$$
\del_1'=\del_1+N_1, \quad \del_2'=\del_2+N_2 
$$ 
for some $N_1, N_2 \in \N$, and 
$$
\calL_{\textbf{a}'}(j)=\calL_{\textbf{a}}(j)-N_1+d_1(j)n_1^j, \quad 
\calL_{\textbf{b}'}(j)=\calL_{\textbf{b}}(j)-N_2+d_2(j)n_2^j, 
$$
where $d_1(\cdot), d_2(\cdot): \N \to \Z$; the result follows from plugging these new quantities appropriately into the above quantities \eqref{20200918eq01}, and we leave the details to the interested reader.

\bigskip

\section{Proof of Theorem \ref{mainresult02}}

The goal of this section is to prove Theorem \ref{mainresult02}, which states that if $\calG_1, \dots, \calG_{d+1}$ are adjacent on $\R^d$, then their bases are closely related. 

\begin{defn}
Let $\{n_1, \dots, n_{d+1}\}$ be a collection of integers with $n_i \ge 2, i=1, \dots, d+1$. We call such a collection \emph{fine} if there exists an adjacent collection of general dyadic systems  $\{\calG_1, \dots, \calG_{d+1}\}$ on $\R^d$, where $\calG_i=\calG(n_i, \del_i, \calL_{\bfVa_i}), i=1, \dots, d+1$.
\end{defn}

Note that Theorem \ref{mainresult02} then asserts that if $\{n_1, \dots, n_{d+1}\}$ is fine, then there exists an integer $\mathfrak n \ge 2$, and $s_i \in \N, s_i \ge 1, i=1, \dots, d+1$, such that
$$
n_i=\mathfrak n^{s_i}, \quad i=1, \dots, d+1. 
$$

We begin with the case when $d=1$. 

\begin{prop} \label{20200922prop01}
If $\{n_1, n_2\}$ is fine, then there exists some $\mathfrak n \in \N$, $\mathfrak n \ge 2$, such that 
$$
n_i=\mathfrak n^{s_i}, \quad i=1, 2,
$$
for some integers $s_1, s_2 \ge 1$. 
\end{prop}

\begin{proof}
We prove by contradiction by assuming such an $\mathfrak n$ does not exist, which means that $\frac{\log n_2}{\log n_1}$ is irrational (see the related comments following the theorem statement in the Introduction).  We will crucially use this fact in the proof. 

Since $\{n_1, n_2\}$ is fine, by Theorem \ref{mainresult01}, for any $k_1, k_2 \in \Z$, we know the estimate \eqref{20200922eq01} holds, in particular, this implies that there exists some $C>0$, such that for any $m \ge 0$ and $k_1, k_2 \in \Z$, 
\begin{equation} \label{20200922eq02} 
\left|\del n_2^m n_1^{\phi_{n_2; n_1}(m)}-k_1n_2^m-k_2n_1^{\phi_{n_2; n_1}(m)} \right| \ge C n_1^{\phi_{n_2; n_1}(m)},
\end{equation}
where we denote $\del:=\del_1-\del_2$. Let us consider the following set
$$
\mathbb A(n_1, n_2; m):=\left\{ k_1 n_2^m+k_2 n_1^{\phi_{n_2; n_1}(m)}: k_1, k_2 \in \Z \right\}.
$$
Note that 
$$
\mathbb A(n_1, n_2; m)=\left\{ k \cdot \gcd \left(n_2^m, n_1^{\phi_{n_2; n_1}(m)}\right): k \in \Z \right\}, 
$$
which suggests that we can find $K_1, K_2 \in \Z$ such that 
\begin{equation} \label{20200922eq04}
\left|\del n_2^m n_1^{\phi_{n_2; n_1}(m)}-K_1n_2^m-K_2n_1^{\phi_{n_2; n_1}(m)} \right| \le \gcd \left(n_2^m, n_1^{\phi_{n_2; n_1}(m)}\right).
\end{equation} 
To use this, let us write 
$$
n_1=p_1^{a_1}p_2^{a_2} \dots p_L^{a_L}
$$
and 
$$
n_2=p_1^{a'_1}p_2^{a'_2} \dots p_L^{a'_L},
$$
where $\{p_1, \dots, p_L\}$ is a finite collection of primes and $a_\ell, a'_\ell \in \N, \ell=1, \dots, L$. Note that the $a_i$'s and $a'_i$'s only depend on $n_1$ and $n_2$, and independent of $m$.

We now consider two different cases. 

\medskip

\textit{Case I:} There exists infinitely many $m \in \N$ and some $\ell \in \{1, \dots, L\}$ such that
\begin{enumerate}
    \item [(1).] $a_\ell \ge 1$;
    
    \medskip
    
    \item [(2).] $p_\ell^{a_\ell \phi_{n_2; n_1}(m)} \ge p_\ell^{a_\ell'm}$.
\end{enumerate}
We make a remark that here one may presumably assume that $\ell$ may depend on the choice of $m$. However, since there are only finitely many choices for $\ell$, by pigeonholing, simply pick a fixed $\ell \in \{1, \dots, L\}$ and restrict attention to the sub-sequence of $m$ which satisfy \textit{Case I}. 

For simplicity, denote
$$
\calC_I:=\{m \in \N: m \ \textrm{satisfies the assumption of {\it Case I}}\}.
$$

For {\it Case I}, we note that contribution of the prime $p_\ell$ to the term $\gcd \left(n_2^m, n_1^{\phi_{n_2; n_1}(m)} \right)$ is at most $p_\ell^{a_\ell' m}$. On the other hand, we define a function $\Psi_1: \N \to \Z$ given by
$$
\Psi_1(m)=a_\ell \phi_{n_2; n_1}(m)-a'_\ell m. 
$$
Note that if $m \in \calC_I$, then 
\begin{enumerate}
    \item [(a).] $\Psi_1 (m) \ge 0$;
    
    \medskip
    
    \item [(b).] $\gcd \left(n_2^m, n_1^{\phi_{n_2; n_1}(m)}\right) \le \frac{n_1^{\phi_{n_2; n_1}(m)}}{p_\ell^{\Phi_1(m)}}$. 
\end{enumerate}

\medskip

We have the following \emph{claim}: $\Psi_1(m)$ is unbounded. 

\medskip

\textit{Proof of the claim:}
Since, 
$$
\Psi_1(m)=a_\ell \cdot \left\lfloor \frac{m \log n_2}{\log n_1} \right\rfloor-a_\ell'm
$$
we have 
\begin{equation} \label{20200922eq07}
\left(\frac{a_\ell \log n_2}{\log n_1}- a_\ell' \right) \cdot m-a_\ell \le \Psi_1(m) \le \left(\frac{a_\ell \log n_2}{\log n_1}- a_\ell' \right) \cdot m
\end{equation} 
however, since $\frac{\log n_2}{\log n_1}$ is irrational, using assertion (a), we can indeed conclude that
$$
\frac{a_\ell \log n_2}{\log n_1}- a_\ell'>0.
$$
This, together with the estimate \eqref{20200922eq07}, clearly implies the desired claim.  \hfill\fbox

\bigskip

By (a) and the claim above, there exists a $m \in \calC_I$ sufficiently large, such that
\begin{equation} \label{20200922eq10}
p_\ell^{-\Psi_1(m)}<\frac{C}{2}, 
\end{equation} 
where we recall that $C$ is defined in \eqref{20200922eq02}. Therefore, we have
\begin{eqnarray*}
\left|\del n_2^m n_1^{\phi_{n_2; n_1}(m)}-K_1n_2^m-K_2n_1^{\phi_{n_2; n_1}(m)} \right|%
&\le& \gcd \left(n_2^m, n_1^{\phi_{n_2; n_1}(m)}\right) \\
&\le& \frac{n_1^{\phi_{n_2; n_1}(m)}}{p_\ell^{\Psi_1(m)}} \\
&<& \frac{Cn_1^{\phi_{n_2; n_1}(m)}}{2},
\end{eqnarray*}
where in the first line above, we use \eqref{20200922eq04}, in the second to last estimate, we use assertion (b) above and in the last estimate, we use \eqref{20200922eq10}. This clearly contradicts \eqref{20200922eq02}. 

\bigskip

\textit{Case II:} Suppose {\it Case I} fails. This means there exists infinitely many $m \in \N$ and some $s \in \{1, \dots, L\}$ (independent of the choice of $m$), such that
\begin{enumerate}
    \item [(i).] $a_s' \ge 1$;
    
    \medskip
    
    \item [(ii).] $p_s^{a_s'm} \ge p_s^{a_s \phi_{n_2; n_1}(m)}$. 
\end{enumerate}
Similarly, we denote 
$$
\calC_{II}:=\left\{ m \in \N: m \ \textrm{satisfies the assumption of \textit{Case II}}\right\}. 
$$
The proof for the second case is similar to the first one, and we only sketch it here. We define $\Psi_2: \N \to \Z$ by 
$$
\Psi_2(m):=a_s'm-a_s \phi_{n_2; n_1}(m)
$$
Note that if $m \in \calC_{II}$, then
\begin{enumerate}
    \item [(c).] $\Psi_2(m) \ge 0$;
    
    \medskip
    
    \item [(d).] $\gcd \left(n_2^m, n_1^{\phi_{n_2; n_1}(m)} \right) \le \frac{n_2^m}{p_s^{\Psi_2(m)}}$.
\end{enumerate}
Similar as above, we can show that $\Psi_2(m)$ is unbounded. Now to use the assertion (d) above, we rewrite \eqref{20200922eq02} a little bit by
\begin{equation} \label{20200922eq21}
\left| \del n_2^m n_1^{\phi_{n_2; n_1}(m)}-k_1n_2^m-k_2n_1^{\phi_{n_2; n_1}(m)} \right| \ge \widetilde{C}n_2^m.
\end{equation} 
This is because $n_1^{\phi_{n_2; n_1}(m)} \ge \frac{n_2^m}{n_1}$ and we may let $\widetilde{C}=\frac{C}{n_1}$.

Finally, let us take $m \in \calC_{II}$ sufficiently large, such that
\begin{equation} \label{20200922eq22}
p_s^{-\Psi_2(m)} \le \frac{\widetilde{C}}{2}
\end{equation} 
The desired contradiction will then follow from assertion (d), \eqref{20200922eq04}, \eqref{20200922eq21} and \eqref{20200922eq22}. 

The proof is complete. 
\end{proof}

Finally we turn to the proof of Theorem \ref{mainresult02}.

\begin{proof} [Proof of Theorem \ref{mainresult02}]
Theorem \ref{mainresult02} is an easy consequence of Proposition \ref{20200922prop01} and Corollary \ref{20200915cor01}, and we would like to leave the details to the interested reader. 
\end{proof}

\begin{rem}
 Notice that \eqref{20200922eq02} reduces to
 \[
 \left|\del n_2^m n_1^{\phi_{n_2; n_1}(m)}-k \cdot \gcd\left(n_1^{\phi_{n_2; n_1}(m)}, n_2^m \right) \right|, \quad \textrm{for some} \ k \in \Z.
 \]
 If we assume adjacency, this means that 
 $$
 \delta \cdot \textrm{lcm}\left(n_1^{\phi_{n_2; n_1}(m)}, n_2^m\right)-k \neq 0,
 $$ 
 that is, $\delta \neq \frac{k}{\textrm{lcm}\left(n_1^{\phi_{n_2; n_1}(m)}, n_2^m\right)}$, a natural analogue of our previous work \cite{AHJOW}.  Since we now know that $n_1$ and $n_2$ are powers of the same base, this exactly replicates the fact that $\frac{k}{n^m}$ are not $n$-far, which appeared in (\cite{AHJOW}, Corollary 2.10).
\end{rem}

\end{document}